\documentclass[11pt,a4paper,reqno]{amsart}

\usepackage{amssymb,latexsym}
\usepackage{graphicx}
\usepackage{url}		%does nice formatting of URLs

\theoremstyle{plain}
\numberwithin{equation}{section}
\newtheorem{thm}{Theorem}[section]
\newtheorem{theorem}[thm]{Theorem}
\newtheorem{lemma}[thm]{Lemma}

\newtheorem{corollary}[thm]{Corollary}

\begin{document}
	
	\title{Explicit Formulas for the $p$-adic Valuations of Fibonomial Coefficients II}
	\author{Phakhinkon Phunphayap}
	\address{Department of Mathematics, Faculty of Science\\
		Silpakorn University\\
		Ratchamankanai Rd, Nakhon Pathom\\
		Thailand, 73000}
	\email{phakhinkon@gmail.com}
	\author{Prapanpong Pongsriiam\textsuperscript{1}}
	\address{Department of Mathematics, Faculty of Science\\
		Silpakorn University\\
		Ratchamankanai Rd, Nakhon Pathom\\
		Thailand, 73000}
	\email{prapanpong@gmail.com, pongsriiam\_p@silpakorn.edu}
	\thanks{\textsuperscript{1}Prapanpong Pongsriiam is the corresponding author.}
	
	\begin{abstract}
		In this article, we give explicit formulas for the $p$-adic valuations of the Fibonomial coefficients ${p^a n \choose n}_F$ for all primes $p$ and positive integers $a$ and $n$. This is a continuation from our previous article extending some results in the literature, which deal only with $p = 2,3,5,7$ and $a = 1$. Then we use these formulas to characterize the positive integers $n$ such that ${pn \choose n}_F$ is divisible by $p$, where $p$ is any prime which is congruent to $\pm 2 \pmod{5}$.
	\end{abstract}
	
	\maketitle
	
	\section{Introduction}
	
	The \textit{Fibonacci sequence} $(F_n)_{n\geq 1}$ is given by the recurrence relation $F_n = F_{n-1} + F_{n-2}$ for $n\geq 3$ with the initial values $F_1 = F_2 = 1$. For each $m\geq 1$ and $1\leq k \leq m$, the \textit{Fibonomial coefficients} ${m \choose k}_F$ is defined by
	\begin{equation*}
		{m \choose k}_F = \frac{F_1F_2F_3\cdots F_m}{(F_1F_2F_3\cdots F_k)(F_1F_2F_3\cdots F_{m-k})} = \frac{F_{m-k+1}F_{m-k+2}\cdots F_m}{F_1F_2F_3\cdots F_k}.
	\end{equation*}
	Similar to the binomial coefficients, we define ${m \choose k}_F=1$ if $k=0$ and ${m \choose k}_F=0$ if $k>m$, and it is well-known that ${m \choose k}_F$ is always an integer for every $m \geq 1$ and $k\geq 0$.
	
	Recently, there has been an interest in the study of Fibonomial coefficients. For example, Marques and Trojovsk\'y \cite{Mar1,Mar2} determine the integers $n\geq 1$ such that ${pn \choose n}_F$ is divisible by $p$ for $p=2,3$. Then Ballot \cite[Theorem 2]{ballot} extends the Kummer-like theorem of Knuth and Wilf \cite{Knuth} and uses it to generalize the results of Marques and Trojovsk\'{y}. In particular, Ballot \cite[Theorems 3.6, 5.2, and 5.3]{ballot} finds all integers $n$ such that $p\mid {pn \choose n}_U$ for any nondegenerate fundamental Lucas sequence $U$ and $p = 2,3$ and for $p=5,7$ in the case $U=F$. For other recent results on the divisibility properties of the Fibonacci numbers, the Fibonomial or Lucasnomial coefficients, and other generalizations of binomial coefficients, see for example in \cite{ballot2,ballot3,ballot4,KP2,KP3,OP,OP3,P,P2,P3}. Hence the relation $p \mid {p^a n \choose n}_F$ has been studied only in the case $p = 2,3,5,7$ and $a=1$. In this article, we extend the investigation on ${p^a n \choose n}_F$ to the case of any prime $p$ and any positive integer $a$. To obtain such the general result for all $p$ and $a$, the calculation is inevitably long but we try to make it as simple as possible. As a reward, we can easily show in Corollaries \ref{corollary1} and \ref{corollary2} that ${4n \choose n}_F$ is odd if and only if $n$ is a nonnegative power of 2, and ${8n \choose n}_F$ is odd if and only if $n = (1 + 3 \cdot 2^{k})/7$ for some $k \equiv 1 \pmod{3}$.
	
	We organize this article as follows. In Section \ref{sec2}, we give some preliminaries and results which are needed in the proof of the main theorems. In Section \ref{sec3}, we calculate the $p$-adic valuation of ${p^a n \choose n}_F$ for all $a$, $p$, and $n$, and use it to give a characterization of the positive integers $n$ such that ${p^a n \choose n}_F$ is divisible by $p$ where $p$ is any prime which is congruent to $\pm 2 \pmod{5}$. Remark that there also is an interesting pattern in the $p$-adic representation of the integers $n$ such that ${pn \choose n}_F$ is divisible by $p$. The proof is being prepared but it is a bit too long to include in this paper. We are trying to make it simpler and shorter and will publish it in the future. For more information and some recent articles related to the Fibonacci numbers, we refer the readers to \cite{PP,P4,P5,P6} and references therein.
	
	\section{Preliminaries and Lemmas}\label{sec2}
	
	Throughout this article, unless stated otherwise, $x$ is a real number, $p$ is a prime, $a,b,k,m,n,q$ are integers, $m,n \geq 1$, and $q \geq 2$. The \textit{$p$-adic valuation (or $p$-adic order)} of $n$, denoted by $\nu_p (n)$, is the exponent of $p$ in the prime factorization of $n$. In addition, the \textit{order (or the rank) of appearance} of $n$ in the Fibonacci sequence, denoted by $z(n)$, is the smallest positive integer $m$ such that $n\mid F_m$, $\lfloor x\rfloor$ is the largest integer less than or equal to $x$, $\{ x\}$ is the \textit{fractional part} of $x$ given by $\{ x\}=x-\lfloor x\rfloor$, $\lceil x\rceil$ is the smallest integer larger than or equal to $x$, and $a\bmod{m}$ is the least nonnegative residue of $a$ modulo $m$. Furthermore, for a mathematical statement $P$, the Iverson notation $[P]$ is defined by
	\begin{equation*}
		[P] = 
		\begin{cases}
			1, & \text{if $P$ holds}; \\
			0, & \text{otherwise}.
		\end{cases}
	\end{equation*}
	We define $s_q (n)$ to be the sum of digits of $n$ when $n$ is written in base $q$, that is, if $n = (a_k a_{k-1} \ldots a_0)_q = a_k q^k + a_{k-1} q^{k-1} + \cdots + a_0$ where $0\leq a_i < q$ for every $i$, then $s_q (n) = a_k + a_{k-1} + \cdots + a_0$. Next, we recall some well-known and useful results for the reader's convenience.
	\begin{lemma}\label{lemma1}
		Let $p\neq 5$ be a prime. Then the following statements hold.
		\begin{itemize}
			\item[(i)] $n\mid F_m$ if and only if $z(n)\mid m$
			\item[(ii)] $z(p)\mid p+1$ if and only if $p\equiv \pm 2 \pmod{5}$ and $z(p)\mid p-1$, otherwise.
			\item[(iii)] $\gcd (z(p),p)=1$.
		\end{itemize}
	\end{lemma}

	\begin{proof}
		These are well-known. See, for example, in \cite[Lemma 1]{PP} for more details.
	\end{proof}

	\begin{lemma}\label{Legendre}
		{\rm(Legendre's formula)}
		Let $n$ be a positive integer and let $p$ be a prime. Then
		\begin{equation*}
			\nu_p(n!)=\sum_{k=1}^\infty \left\lfloor\frac{n}{p^k}\right\rfloor = \frac{n - s_p (n)}{p-1}.
		\end{equation*}
	\end{lemma}

	We will deal with a lot of calculations involving the floor function. So we recall the following results, which will be used throughout this article, sometimes without reference.
	
	\begin{lemma}\label{floorlemma}
		For $k\in\mathbb{Z}$ and $x\in\mathbb{R}$, the following holds
		\begin{itemize}
			\item[(i)] $\lfloor k+x\rfloor = k + \lfloor x\rfloor$,
			\item[(ii)] $\{k+x\}=\{x\}$,
			\item[(iii)] $\lfloor x\rfloor + \lfloor -x \rfloor=
			\begin{cases}
			-1, &\text{if $x\not\in\mathbb{Z}$;}\\
			0, &\text{if $x\in\mathbb{Z}$,}
			\end{cases}$
			\item[(iv)] $0 \leq \{ x \} < 1$ and $\{ x \} = 0$ if and only if $x \in \mathbb{Z}$.
			\item[(v)] $\lfloor x+y\rfloor =
			\begin{cases}
			\lfloor x\rfloor +\lfloor y\rfloor, &\text{if $\{x\}+\{y\}<1$;}\\
			\lfloor x\rfloor +\lfloor y\rfloor+1, &\text{if $\{x\}+\{y\}\geq 1$,}
			\end{cases}$
			\item[(vi)] $\left\lfloor\frac{\lfloor x\rfloor}{k}\right\rfloor = \left\lfloor\frac{x}{k}\right\rfloor$ for $k\geq 1$.
		\end{itemize}
	\end{lemma}

	\begin{proof}
		These are well-known and can be proved easily. For more details, see in \cite[Chapter 3]{Concrete}. We also refer the reader to \cite{OP2} for a nice application of (vi).
	\end{proof}

	The next three theorems given by Phunphayap and Pongsriiam \cite{PP} are important tools for obtaining the main results of this article.
	
	\begin{theorem}\label{theorem1}
		\cite[Theorem 7]{PP} Let $p$ be a prime, $a\geq 0$, $\ell \geq 0$, and $m\geq 1$. Assume that $p\equiv \pm 1\pmod{m}$ and $\delta = [\ell \not\equiv 0\pmod{m}]$ is the Iverson notation. Then
		\begin{equation*}
		\nu_p \left(\left\lfloor\frac{\ell p^a}{m}\right\rfloor !\right) =
		\begin{cases}
		\frac{\ell (p^a -1)}{m(p-1)}-a\left\{\frac{\ell}{m}\right\} + \nu_p\left(\left\lfloor\frac{\ell}{m}\right\rfloor!\right), &\text{if $p\equiv 1 \pmod{m}$;}\\
		\frac{\ell (p^a -1)}{m(p-1)} - \frac{a}{2}\delta + \nu_p\left(\left\lfloor\frac{\ell}{m}\right\rfloor!\right), &\text{if $p\equiv -1 \pmod{m}$ and $a$ is even;}\\
		\frac{\ell (p^a -1)}{m(p-1)} - \frac{a-1}{2}\delta - \left\{\frac{\ell}{m}\right\} + \nu_p\left(\left\lfloor\frac{\ell}{m}\right\rfloor!\right), &\text{if $p\equiv -1 \pmod{m}$ and $a$ is odd.}
		\end{cases}
		\end{equation*}
	\end{theorem}
	
	\begin{theorem}\label{theorem2}
		\cite[Theorem 11 and Corollary 12]{PP} Let $0\leq k\leq m$ be integers. Then the following statements hold.
			\begin{itemize}
				\item[(i)] Let $A_2 = \nu_2\left(\left\lfloor \frac{m}{6} \right\rfloor!\right) - \nu_2\left(\left\lfloor \frac{k}{6} \right\rfloor!\right) -\nu_2\left(\left\lfloor \frac{m-k}{6} \right\rfloor!\right)$. If $r = m \bmod{6}$ and $s = k \bmod{6}$, then
				\begin{equation*}
 					\nu_2\left({m \choose k}_F\right) = 
					\begin{cases}
					A_2, &\text{if $r\geq s$ and $(r,s)\neq (3,1),(3,2),(4,2)$;}\\
					A_2 + 1, &\text{if $(r,s) = (3,1),(3,2),(4,2)$;}\\
					A_2+3, &\text{if $r<s$ and $(r,s)\neq (0,3),(1,3),(2,3)$,}\\
					&(1,4),(2,4),(2,5);\\
					A_2 + 2, &\text{if $(r,s)= (0,3),(1,3),(2,3),(1,4),(2,4)$,}\\
					& (2,5).
					\end{cases}
				\end{equation*}
				\item[(ii)] $\nu_5\left({m \choose k}_F\right) = \nu_5\left({m \choose k}\right)$.
				\item[(iii)] Suppose that $p$ is a prime, $p\neq 2$, and $p\neq 5$. If $m'=\left\lfloor\frac{m}{z(p)}\right\rfloor$, $k'=\left\lfloor\frac{k}{z(p)}\right\rfloor$, $r=m\bmod z(p)$, and $s=k\bmod z(p)$, then
				\begin{align*}
					\nu_p\left({m \choose k}_F\right) &=  \nu_p \left({m' \choose k'}\right) + [r<s]\left (\nu_p\left(\left\lfloor\frac{m-k+z(p)}{z(p)}\right\rfloor\right) +\nu_p(F_{z(p)})\right ).
				\end{align*}
			\end{itemize}
	\end{theorem}

	\begin{theorem}\label{theorem3}
		\cite[Theorem 13]{PP} Let $a$, $b$, $\ell_1$, and $\ell_2$ be positive integers and $b\geq a$. For each $p \neq 5$, assume that $\ell_1 p^b >\ell_2 p^a$ and let $m_p = \left\lfloor\frac{\ell_1 p^{b-a}}{z(p)}\right\rfloor$ and $k_p =\left\lfloor\frac{\ell_2}{z(p)}\right\rfloor$. Then the following statements hold.
		\begin{itemize}
			\item[(i)] If $a\equiv b \pmod{2}$, then $\nu_2\left( {\ell_1 2^b \choose \ell_2 2^a}_F \right)$ is equal to
			\begin{equation*}
			\begin{cases}
			\nu_2 \left( {m_2 \choose k_2} \right), &\text{if $\ell_1 \equiv \ell_2 \pmod{3}$ or $\ell_2 \equiv 0 \pmod{3}$;}\\
			a + 2 + \nu_2 \left( m_2	- 	k_2 \right) + \nu_2 \left( {m_2 \choose k_2} \right), &\text{if $\ell_1 \equiv 
				0 \pmod{3}$ and $\ell_2 \not\equiv 
				0 \pmod{3}$;}\\
			\left\lceil \frac{a}{2} \right\rceil + 1 + \nu_2 \left( m_2	- k_2 \right) + \nu_2 \left( { m_2 \choose  k_2} \right), &\text{if $\ell_1 \equiv 1 \pmod{3}$ and $\ell_2 \equiv 2 \pmod{3}$;}\\
			\left\lceil\frac{a+1}{2}\right\rceil + \nu_2 \left( {m_2 \choose k_2} \right), &\text{if $\ell_1 \equiv 2 \pmod{3}$ and $\ell_2 \equiv 1 \pmod{3}$,}
			\end{cases}
			\end{equation*}
			and if $a\not\equiv b \pmod{2}$, then $\nu_2\left( {\ell_1 2^b \choose \ell_2 2^a}_F \right)$ is equal to
			\begin{equation*}
			\begin{cases}
			\nu_2 \left( {m_2 \choose k_2} \right), &\text{if $\ell_1 \equiv  -\ell_2 \pmod{3}$ or $\ell_2 \equiv 0 \pmod{3}$;}\\
			a + 2 + \nu_2 \left( m_2	- k_2 \right) + \nu_2 \left( {m_2 \choose k_2} \right), &\text{if $\ell_1 \equiv 
				0 \pmod{3}$ and $\ell_2 \not\equiv 
				0 \pmod{3}$;}\\
			\left\lceil \frac{a+1}{2} \right\rceil + \nu_2 \left( {m_2 \choose k_2} \right), &\text{if $\ell_1 \equiv 
				1 \pmod{3}$ and $\ell_2 \equiv 
				1 \pmod{3}$;}\\
			\left\lceil\frac{a}{2}\right\rceil + 1 + \nu_2 \left( m_2	- k_2 \right) + \nu_2 \left( {m_2 \choose k_2} \right), &\text{if $\ell_1 \equiv 2 \pmod{3}$ and $\ell_2 \equiv 2 \pmod{3}$.}
			\end{cases}
			\end{equation*}
			\item[(ii)] Let $p\neq 5$ be an odd prime and let $r=\ell_1 p^b \mod{z(p)}$ and $s=\ell_2 p^a \mod{z(p)}$. If $p\equiv \pm 1\pmod{5}$, then
			\begin{equation*}
			\nu_p\left( {\ell_1 p^b \choose \ell_2 p^a}_F \right) = [r<s]\left(a + \nu_p \left( m_p	- k_p \right) + \nu_p(F_{z(p)}) \right) + \nu_p \left( {m_p \choose k_p} \right),
			\end{equation*} 
			and if $p\equiv \pm 2\pmod{5}$, then $\nu_p\left( {\ell_1 p^b \choose \ell_2 p^a}_F \right)$ is equal to
			\begin{equation*}
			\begin{cases}
			\nu_p \left( {m_p \choose k_p} \right), &\text{if $r=s$ or $\ell_2 \equiv 0 \pmod{z(p)}$;}\\
			a + \nu_p (F_{z(p)}) + \nu_p \left( m_p	- 	k_p \right) + \nu_p \left( {m_p \choose k_p} \right), &\text{if $\ell_1 \equiv 0 \pmod{z(p)}$ and}\\
			& \ell_2 \not\equiv 0 \pmod{z(p)};\\
			\frac{a}{2} + \nu_p \left( {m_p \choose k_p} \right), &\text{if $r> s$, $\ell_1,\ell_2 \not\equiv 0\pmod{z(p)}$,}\\
			&\text{and $a$ is even;}\\ 
			\frac{a}{2} + \nu_p (F_{z(p)}) + \nu_p \left( m_p	- 	k_p \right) + \nu_p \left( {m_p \choose k_p} \right), &\text{if $r<s$, $\ell_1,\ell_2 \not\equiv 0 \pmod{z(p)}$,}\\
			&\text{and $a$ is even;}\\
			\frac{a+1}{2} + \nu_p \left( m_p	- k_p \right) + \nu_p \left( {m_p \choose k_p} \right), &\text{if $r>s$, $\ell_1,\ell_2 \not\equiv 0\pmod{z(p)}$,}\\
			&\text{and $a$ is odd;}\\
			\frac{a-1}{2} + \nu_p(F_{z(p)}) + \nu_p \left( {m_p \choose k_p} \right), &\text{if $r<s$, $\ell_1,\ell_2 \not\equiv 0\pmod{z(p)}$,}\\
			&\text{and $a$ is odd.}
			\end{cases}
			\end{equation*}	
		\end{itemize} 
	\end{theorem}
	
	In fact, Phunphayap and Pongsriiam \cite{PP} obtain other results analogous to Theorems \ref{theorem2} and \ref{theorem3} too but we do not need them in this article.
	
	\section{Main Results}\label{sec3}
	
	We begin with the calculation of the 2-adic valuation of ${2^a n \choose n}_F$ and then use it to determine the integers $n$ such that ${2n \choose n}_F,{4n \choose n}_F,{8n \choose n}_F$ are even. Then we calculate the $p$-adic valuation of ${p^a n \choose n}_F$ for all odd primes $p$. For binomial coefficients, we know that $\nu_2 \left({2n \choose n}\right) = s_2 (n)$. For Fibonomial coefficients, we have the following result.
	
	\begin{theorem}\label{maintheorem0}
		Let $a$ and $n$ be positive integers, $\varepsilon = [n \not \equiv 0 \pmod{3}]$, and $A = \left\lfloor \frac{(2^a - 1)n}{3\cdot 2^{\nu_2 (n)}} \right\rfloor$. Then the following statements hold.
		\begin{itemize}
			\item[(i)] If $a$ is even, then 
			\begin{equation}\label{eq0thm0}
				\nu_2 \left( {2^a n \choose n}_F \right) = \delta + A - \frac{a}{2}\varepsilon  - \nu_2 (A!) = \delta + s_2 (A) - \frac{a}{2}\varepsilon,
			\end{equation} 
			where $\delta = [n \bmod{6} = 3,5]$. In other words, $\delta = 1$ if $n \equiv 3,5 \pmod{6}$ and $\delta = 0$ otherwise.
			\item[(ii)] If $a$ is odd, then 
			\begin{equation}\label{eq00thm0}
				 \nu_2 \left( {2^a n \choose n}_F \right) = \delta + A - \frac{a-1}{2}\varepsilon - \nu_2 (A!) = \delta + s_2 (A) - \frac{a-1}{2}\varepsilon,
			\end{equation} 
			where $\delta = \frac{(n\bmod{6}) - 1}{2} [2\nmid n] + \left\lceil \frac{\nu_2 (n) + 3 - n \bmod{3}}{2} \right\rceil [n \bmod 6 = 2,4]$. In other words, $\delta = \frac{(n \bmod{6}) - 1}{2}$ if $n$ is odd, $\delta = 0$ if $n \equiv 0 \pmod{6}$, $\delta = \left\lceil \frac{\nu_2 (n)}{2} \right\rceil + 1$ if $n \equiv 4 \pmod{6}$, and $\delta = \left\lceil \frac{\nu_2 (n) + 1}{2} \right\rceil $ if $n \equiv 2 \pmod{6}$.
		\end{itemize}
	\end{theorem}
	
	\begin{proof}
		The second equalities in \eqref{eq0thm0} and \eqref{eq00thm0} follow from Legendre's formula. So it remains to prove the first equalities in \eqref{eq0thm0} and \eqref{eq00thm0}. To prove (i), we suppose that $a$ is even and divide the consideration into two cases.
		
		\noindent{\bf Case 1.} $2 \nmid n$. Let $r = 2^a n \bmod{6}$ and $s = n \bmod{6}$. Then $s \in \{1,3,5\}$, $r \equiv 2^a n \equiv 4n \equiv 4s \pmod{6}$, and therefore $(r,s) = (4,1),(0,3),(2,5)$. In addition, $A = \left\lfloor \frac{(2^a - 1)n}{3} \right\rfloor = \frac{(2^a - 1)n}{3}$ and $\delta = [s = 3,5]$. By Theorem \ref{theorem2}(i), the left--hand side of \eqref{eq0thm0} is $A_2$ if $s = 1$ and $A_2 + 2$ if $s = 3,5$, where $A_2 = \nu_{2}\left( \left\lfloor \frac{2^a n}{6} \right\rfloor ! \right) - \nu_2 \left( \left\lfloor \frac{n}{6} \right\rfloor ! \right) - \nu_2 \left( \left\lfloor \frac{(2^a - 1)n}{6} \right\rfloor ! \right)$. We obtain by Theorem \ref{theorem1} that 
		\begin{equation*}
			\nu_2 \left( \left\lfloor \frac{2^a n}{6} \right\rfloor ! \right) = \nu_2 \left( \left\lfloor \frac{2^{a-1}n}{3} \right\rfloor ! \right) = \frac{(2^{a-1} - 1)n}{3} - \frac{a-2}{2}\varepsilon - \left\{ \frac{n}{3} \right\} + \nu_2 \left( \left\lfloor \frac{n}{3} \right\rfloor ! \right).
		\end{equation*}
		By Legendre's formula and Lemma \ref{floorlemma}, we have
		\begin{equation*}
			\nu_2 \left( \left\lfloor \frac{n}{6} \right\rfloor ! \right) = \nu_2 \left( \left\lfloor \frac{n}{3} \right\rfloor ! \right) - \left\lfloor \frac{n}{6} \right\rfloor,
		\end{equation*}
		\begin{equation*}
			\nu_2 \left( \left\lfloor \frac{(2^a - 1)n}{6} \right\rfloor ! \right) = \nu_2 \left( \left\lfloor \frac{(2^a - 1)n}{3} \right\rfloor ! \right) - \left\lfloor \frac{(2^a - 1)n}{6} \right\rfloor = \nu_2 (A!) - \left\lfloor \frac{(2^a - 1)n}{6} \right\rfloor, 
		\end{equation*}
		\begin{equation*}
			\left\lfloor \frac{n}{6} \right\rfloor + \left\lfloor \frac{(2^a - 1)n}{6} \right\rfloor = \frac{n - s}{6} + \frac{2^a n - r}{6} - \frac{n - s}{6} + \left\lfloor \frac{r - s}{6} \right\rfloor = \frac{2^a n - r}{6} - [s \in \{ 3,5 \}].
		\end{equation*}
		From the above observation, we obtain 
		\begin{align*}
			A_2 &= \frac{(2^{a-1} - 1)n}{3} - \frac{a-2}{2}\varepsilon - \left\{ \frac{n}{3} \right\} + \frac{2^a n - r}{6} - [s\in\{ 3,5 \}] - \nu_2 (A!) \\
			&= A - \frac{a-2}{2}\varepsilon - \left\{\frac{n}{3} \right\} - \frac{r}{6} - [s \in \{3,5\}] - \nu_2 (A!) \\
			&=
			\begin{cases}
				A - \frac{a}{2} - \nu_2 (A!), &\text{if $s = 1$;} \\
				A - \nu_2 (A!) - 1, &\text{if $s = 3$;} \\
				A - \frac{a}{2} - \nu_2 (A!) - 1, &\text{if $s = 5$.}
			\end{cases}
		\end{align*}
		It is now easy to check that $A_2$ (if $s=1$), $A_2 + 2$ (if $s=3,5$) are the same as $\delta + A - \frac{a}{2}\varepsilon - \nu_2 (A!)$ in \eqref{eq0thm0}. So \eqref{eq0thm0} is verified.
		
		\noindent{\bf Case 2.} $2 \mid n$. We write $n = 2^b \ell$ where $2 \nmid \ell$ and  let $m = \left\lfloor \frac{2^a \ell}{3} \right\rfloor$, $k = \left\lfloor \frac{\ell}{3} \right\rfloor$, $r = 2^a \ell \bmod{3}$, and $s = \ell \bmod{3}$. Since $a$ is even, $r = s$. Then we apply Theorem \ref{theorem3}(i) to obtain 
		\begin{equation}\label{eq2thm0}
			\nu_2 \left( {2^a n \choose n}_F \right) = \nu_2 \left( {\ell 2^{a+b} \choose \ell 2^b}_F \right) = \nu_2 \left( {m \choose k} \right) = \nu_2 (m!) - \nu_2 (k!) - \nu_2 ((m-k)!).
		\end{equation} 
		We see that $\ell \not\equiv 0 \pmod{3}$ if and only if $n \not\equiv 0\pmod{3}$. In addition, $A = \frac{(2^a - 1)\ell}{3}$ and $\delta = 0$. By Theorem \ref{theorem1}, we have
		\begin{equation*}
			\nu_2 (m!) = A - \frac{a}{2} \varepsilon + \nu_2 (k!).
		\end{equation*}
		In addition,
		\begin{equation*}
			m - k = \left\lfloor \frac{2^a \ell}{3} \right\rfloor - \left\lfloor \frac{\ell}{3} \right\rfloor = \frac{2^a \ell - r}{3} - \frac{\ell - s}{3} = \frac{2^a \ell - \ell}{3} = A.
		\end{equation*}
		So $\nu_2 ((m-k)!) = \nu_2 (A!)$. Substituting these in \eqref{eq2thm0}, we obtain \eqref{eq0thm0}. This completes the proof of (i).
		
		To prove (ii), we suppose that $a$ is odd and divide the proof into two cases.
		
		\noindent{\bf Case 1.} $2\nmid n$. This case is similar to Case 1 of the previous part. So we let $r = 2^a n \bmod{6}$ and $s = n \bmod{6}$. Then $s \in \{1,3,5\}$, $r \equiv 2^a n \equiv 2n \equiv 2s \pmod{6}$, $(r,s) = (2,1), (0,3),(4,5)$, $\delta = \frac{s-1}{2}$, and the left--hand side of \eqref{eq00thm0} is $A_2$ if $s = 1$, $A_2 + 2 $ if $s = 3$, and $A_2 + 3$ if $s = 5$, where $A_2 = \nu_2 \left( \left\lfloor \frac{2^a n}{6} \right\rfloor ! \right) - \nu_2 \left( \left\lfloor \frac{n}{6} \right\rfloor ! \right) - \nu_2 \left( \left\lfloor \frac{(2^a - 1)n}{6} \right\rfloor ! \right)$. In addition, we have
		\begin{equation*}
			\nu_2 \left( \left\lfloor \frac{2^a n}{6} \right\rfloor ! \right) = \frac{(2^{a-1} - 1)n}{3} - \frac{a - 1}{2}\varepsilon + \nu_2 \left( \left\lfloor\frac{n}{3}\right\rfloor ! \right),
		\end{equation*}
		\begin{equation*}
			\nu_2 \left( \left\lfloor \frac{n}{6} \right\rfloor ! \right) = \nu_2 \left( \left\lfloor \frac{n}{3} \right\rfloor ! \right) - \left\lfloor \frac{n}{6} \right\rfloor,
		\end{equation*}
		\begin{equation*}
			\nu_2 \left( \left\lfloor \frac{(2^a - 1)n}{6} \right\rfloor ! \right) = \nu_2 (A!) - \left\lfloor \frac{(2^a - 1)n}{6} \right\rfloor,
		\end{equation*}
		\begin{equation*}
			\left\lfloor \frac{n}{6} \right\rfloor + \left\lfloor \frac{(2^a - 1)n}{6} \right\rfloor = \frac{2^a n - r}{6} - [s \in \{3,5\}].
		\end{equation*}
		Therefore 
		\begin{equation*}
			A_2 = \frac{(2^{a-1} - 1)n}{3} - \frac{a-1}{2}\varepsilon + \frac{2^a n - r}{6} - [s\in\{ 3,5 \}] - \nu_2 (A!).
		\end{equation*}
		Furthermore,
		\begin{equation*}
			A = \left\lfloor \frac{(2^a - 1)n}{3} \right\rfloor = \frac{2^a n - r}{3} - \frac{n - s}{3} + \left\lfloor \frac{r - s}{3} \right\rfloor =
			\begin{cases}
				\frac{(2^a - 1)n}{3} - \frac{1}{3}, &\text{if $s = 1$;}\\
				\frac{(2^a - 1)n}{3}, &\text{if $s = 3$;}\\
				\frac{(2^a - 1)n}{3} - \frac{2}{3}, &\text{if $s = 5$,}
			\end{cases}
		\end{equation*}
		which implies that $A = \frac{(2^a - 1)n}{3} - \frac{r}{6}$. Then
		\begin{equation*}
			 A_2 = A - \frac{a-1}{2}\varepsilon - [s \in \{3,5\}] - \nu_2 (A!).
		\end{equation*} 
		It is now easy to check that $A_2$ (if $s=1$), $A_2 + 2$ (if $s =3$), and $A_2 +3$ (if $s = 5$), are the same as $\delta + A - \frac{a-1}{2}\varepsilon  - \nu_2 (A!)$ in \eqref{eq00thm0}. So \eqref{eq00thm0} is verified.
		
		\noindent{\bf Case 2.} $2\mid n$. This case is similar to Case 2 of the previous part. So we write $n = 2^{b} \ell$ where $2 \nmid \ell$ and let $m = \left\lfloor \frac{2^a \ell}{3} \right\rfloor$, $k = \left\lfloor \frac{\ell}{3} \right\rfloor$, $r = 2^a \ell \bmod{3}$, and $s = \ell \bmod{3}$. We obtain by Theorem \ref{theorem3} that $\nu_2 \left( {2^a n \choose n}_F \right)$ is equal to
		\begin{equation}\label{eq5thm0}
			\nu_2 \left( {\ell 2^{a+b} \choose \ell 2^{b}}_F \right) = 
			\begin{cases}
				\nu_2 \left( {m \choose k} \right), &\text{if $\ell \equiv 0 \pmod{3}$;}\\
				\left\lceil \frac{b + 1}{2} \right\rceil + \nu_2 \left( {m \choose k} \right), &\text{if $\ell \equiv 1 \pmod{3}$;}\\
				\left\lceil \frac{b}{2} \right\rceil + 1 + \nu_2 (m - k) + \nu_2 \left( {m \choose k} \right), &\text{if $\ell \equiv 2 \pmod{3}$.}
			\end{cases}
		\end{equation}
		By Theorem \ref{theorem1}, we have
		\begin{equation*}
			\nu_2 (m!) = \frac{(2^a - 1)\ell}{3} - \frac{a - 1}{2} \varepsilon - \left\{ \frac{\ell}{3} \right\} + \nu_2 (k!).
		\end{equation*}
		Since $(2^a - 1)\ell \equiv \ell \pmod{3}$, $\left\{ \frac{(2^a - 1)\ell }{3} \right\} = \left\{ \frac{\ell}{3} \right\}$. This implies that $\nu_2 (m!) = A - \frac{a-1}{2}\varepsilon + \nu_2(k!)$. In addition, $(r,s) = (0,0), (2,1),(1,2)$, and
		\begin{equation*}
			m - k = \left\lfloor \frac{2^a \ell}{3} \right\rfloor - \left\lfloor \frac{\ell}{3} \right\rfloor = \frac{2^a \ell - r}{3} - \frac{\ell - s}{3} = \frac{(2^a - 1)\ell - (r-s)}{3} = A + [s = 2].
		\end{equation*}
		From the above observation, we obtain
		\begin{equation*}
			\nu_2 \left( {m \choose k} \right) = \nu_2 (m!) - \nu_2 (k!) - \nu_2 ((m-k)!) = 
			\begin{cases}
				A - \frac{a - 1}{2} \varepsilon - \nu_2 (A!), &\text{if $s=0,1$;}\\
				A - \frac{a - 1}{2} \varepsilon - \nu_2 ((A+1)!), &\text{if $s=2$.}
			\end{cases}
		\end{equation*}
		Substituting this in \eqref{eq5thm0}, we see that
		\begin{equation}\label{eq6thm0}
			\nu_2 \left( {2^an \choose n}_F \right) = 
			\begin{cases}
				A - \nu_2 (A!), &\text{if $\ell \equiv 0 \pmod{3}$;}\\
				\left\lceil \frac{b + 1}{2} \right\rceil + A - \frac{a - 1}{2} - \nu_2 (A!), &\text{if $\ell \equiv 1 \pmod{3}$;}\\
				\left\lceil \frac{b}{2} \right\rceil + 1 + A - \frac{a - 1}{2} - \nu_2 (A!), &\text{if $\ell \equiv 2 \pmod{3}$.}
			\end{cases}
		\end{equation}
		Recall that $n = 2^b \ell \equiv (-1)^b \ell \pmod{3}$. So \eqref{eq6thm0} implies that
		\begin{equation*}
			\nu_2 \left( {2^an \choose n}_F \right) = 
			\begin{cases}
				A - \nu_2 (A!), &\text{if $n \equiv 0 \pmod{3}$;}\\
				\frac{b}{2} + 1 + A - \frac{a-1}{2} - \nu_2 (A!), &\text{if $n \equiv 1 \pmod{3}$ and $b$ is even;}\\
				\frac{b+1}{2} + 1 + A - \frac{a-1}{2} - \nu_2 (A!), &\text{if $n \equiv 1 \pmod{3}$ and $b$ is odd;}\\
				\frac{b}{2} + 1 + A - \frac{a-1}{2} - \nu_2 (A!), &\text{if $n \equiv 2 \pmod{3}$ and $b$ is even;}\\
				\frac{b+1}{2} + A - \frac{a-1}{2} - \nu_2 (A!), &\text{if $n \equiv 2 \pmod{3}$ and $b$ is odd,}\\
			\end{cases}
		\end{equation*}
		which is the same as \eqref{eq00thm0}. This completes the proof.
	\end{proof}

	We can obtain the main result of Maques and Trojovsk\'y \cite{Mar1} as a corollary.

	\begin{corollary}
		{\rm(Marques and Trojovsk\'y \cite{Mar1})} 
		${2n \choose n}_F$ is even for all $n \geq 2$.
	\end{corollary}

	\begin{proof}
		Let $n \geq 2$ and apply Theorem \ref{maintheorem0} with $a = 1$ to obtain $\nu_2 \left( {2n \choose n}_F \right) = \delta + s_2 (A)$. If $n \not \equiv 0,1 \pmod{6}$, then $\delta >0$. If $n \equiv 0 \pmod{6}$, then $n \geq 3\cdot 2^{\nu_2 (n)}$, and so $A \geq 1$ and $s_2 (A) > 0$. If $n \equiv 1 \pmod{6}$, then $A = \left\lfloor \frac{n}{3} \right\rfloor > 1$ and so $s_2 (A) > 0$. In any case, $\nu_2 \left( {2n \choose n}_F \right) > 0$. So ${2n \choose n}_F$ is even.
	\end{proof}

	\begin{corollary}\label{corollary1}
		Let $n \geq 2$. Then ${4n \choose n}_F$ is even if and only if $n$ is not a power of 2. In other words, for each $n \in \mathbb{N}$, ${4n \choose n}_F$ is odd if and only if $n = 2^k$ for some $k\geq 0$.
	\end{corollary}

	\begin{proof}
		Let $\delta$, $\varepsilon$, and $A$ be as in Theorem \ref{maintheorem0}. If $n = 2^k$ for some $k \geq 1$, then we apply Theorem \ref{maintheorem0} with $a = 2$, $\delta = 0$, $\varepsilon = 1$, $A = 1$ leading to $\nu_2 \left( {4n \choose n}_F \right) = 0$, which implies that ${4n \choose n}_F$ is odd.
		
		Suppose $n$ is not a power of 2. By Theorem \ref{maintheorem0}, $\nu_2 \left( {4n \choose n}_F \right) = \delta + s_2 (A) - \varepsilon \geq s_2 (A) - 1$. Since $n$ is not a power of 2, the sum $s_2 (n) \geq 2$. It is easy to see that $s_2 (m) = s_2 (2^c m)$ for any $c,m \in \mathbb{N}$. Therefore $s_2 (A) = s_2 \left( \frac{n}{2^{\nu_2 (n)}} \right) = s_2 \left(2^{\nu_2 (n)} \cdot \frac{n}{2^{\nu_2 (n)}} \right) = s_2 (n) \geq 2$, which implies $\nu_2 \left( {4n \choose n}_F \right) \geq 1$, as required.
	\end{proof}

	Observe that $2,2^2,2^3$ are congruent to $2,4,1 \pmod{7}$, respectively. This implies that if $k \geq 1$ and $k \equiv 1 \pmod{3}$, then $(1 + 3 \cdot 2^k)/7$ is an integer. We can determine the integers $n$ such that ${8n \choose n}_F$ is odd as follows.
	
	\begin{corollary}\label{corollary2}
		${8n \choose n}_F$ is odd if and only if $n = \frac{1 + 3 \cdot 2^k}{7}$ for some $k \equiv 1 \pmod{3}$.
	\end{corollary}

	\begin{proof}
		Let $a,\delta, A, \varepsilon$ be as in Theorem \ref{maintheorem0}. We first suppose $n = (1 + 3 \cdot 2^k)/7$ where $k \geq 1$ and $k \equiv 1 \pmod{3}$. Then $n \equiv 7n \equiv 1 + 3 \cdot 2^k \equiv 1 \pmod{6}$. Then $a=3$, $\varepsilon = 1$, $\delta = 0$, $A = 2^k$, and so $\nu_{2} \left( {8n \choose n}_F \right) = 0$. Therefore ${8n \choose n}_F$ is odd. Next, assume that ${8n \choose n}_F$ is odd. Observe that $A \geq 2$ and $s_2 (A) > 0$. If $n \equiv 0 \pmod{3}$, then $\varepsilon = 0$ and $\nu_2 \left( {8n \choose n}_F \right) = \delta + s_2(A) > 0$, which is not the case. Therefore $n \equiv 1,2 \pmod{3}$, and so $\varepsilon = 1$. If $n\equiv 0 \pmod{2}$, then $\delta = \left\lceil \frac{\nu_2 (n) + 3 - n \bmod{3}}{2} \right\rceil \geq 1$, and so $\left( {8n \choose n} \right)_F \geq s_2 (A) >0$, which is a contradiction. So $n \equiv 1 \pmod{2}$. This implies $n \equiv 1,5 \pmod{6}$. But if $n \equiv 5 \pmod{6}$, then $\delta \geq 2$ and $\nu_2 \left( {8n \choose n}_F \right) > 0$, a contradiction. Hence $n \equiv 1 \pmod{6}$. Then $\delta = 0$. Since $s_2 (A) - 1 = \nu_2 \left( {8n \choose n}_F \right) = 0$, we see that $A = 2^k$ for some $k \geq 1$. Then $\frac{7n - 1}{3} = \left\lfloor \frac{7n}{3} \right\rfloor = A = 2^k$, which implies  $n = \frac{1 + 3 \cdot 2^k}{7}$, as required.
	\end{proof}
	
	\begin{theorem}\label{theoremp5}
		For each $a,n \in \mathbb{N}$, $\nu_5 \left( {5^a n \choose n}_F \right) = \nu_5 \left( {5^a n \choose n} \right) = \frac{s_5 ((5^a - 1)n)}{4}$. In particular, ${5^an \choose n}_F$ is divisible by 5 for every $a,n \in \mathbb{N}$.
	\end{theorem}

	\begin{proof}
		The first equality follows immediately from Theorem \ref{theorem2}(ii). By Legendre's formula, $\nu_5 \left( {n \choose k} \right) = \frac{s_5 (k) + s_5 (n - k) - s_5 (n)}{4}$ for all $n \geq k \geq 1$. So $\nu_5 \left( {5^a n \choose n}_F \right)$ is
		\begin{equation*}
			\frac{s_5 (n) + s_5 (5^a n - n) - s_5 (5^a n)}{4} = \frac{s_5 ((5^a - 1)n)}{4}. \qedhere
		\end{equation*}
	\end{proof}

	\begin{theorem}\label{maintheorem1}
		Let $p \neq 2,5$, $a, n \in \mathbb{N}$, $r = p^a n \bmod{z(p)}$, $s = n \bmod{z(p)}$, and $A = \left\lfloor \frac{n(p^a - 1)}{p^{\nu_p (n)}z(p)} \right\rfloor$. Then the following statements hold.
		\begin{itemize}
			\item[(i)] If $p \equiv \pm 1 \pmod{5}$, then $\nu_p \left( {p^a n \choose n}_F \right)$ is equal to
			\begin{equation}\label{eq0thm1}
				\frac{A}{p-1} - a\left\{ \frac{n}{p^{\nu_p (n)}z(p)} \right\} - \nu_p (A!) = \frac{s_p (A)}{p - 1} - a\left\{ \frac{n}{p^{\nu_p (n)}z(p)} \right\}.
			\end{equation}
			
			\item[(ii)] If $p \equiv \pm 2 \pmod{5}$ and $a$ is even, then $\nu_p \left( {p^a n \choose n}_F \right)$ is equal to
			\begin{equation}\label{eq00thm1}
				 \frac{A}{p-1} - \frac{a}{2}[s \neq 0] - \nu_p (A!) = \frac{s_p (A)}{p - 1} - \frac{a}{2}[s \neq 0].
			\end{equation}
			
			\item[(iii)] If $p \equiv \pm 2 \pmod{5}$ and $a$ is odd, then $\nu_p \left( {p^a n \choose n}_F \right)$ is equal to
			\begin{equation}\label{eq000thm1}
				\left\lfloor \frac{A}{p - 1} \right\rfloor - \frac{a-1}{2}[s \neq 0] - \nu_p (A!) + \delta,
			\end{equation}
			where $\delta = \left( \left\lfloor \frac{\nu_p (n)}{2} \right\rfloor + [2 \nmid \nu_p(n)][r>s] + [r < s]\nu_p(F_{z(p)}) \right)[r \neq s]$, or equivalently, $\delta = 0$ if $r = s$, $\delta = \left\lfloor \frac{\nu_p (n)}{2} \right\rfloor + \nu_p (F_{z(p)})$ if $r < s$, and $\delta = \left\lceil \frac{\nu_p (n)}{2} \right\rceil$ if $r > s$.
		\end{itemize}
	\end{theorem}
		
	\begin{proof}
		We first prove (i) and (ii). So we suppose that the hypothesis of (i) or (ii) is true. By writing $\nu_p (A!) = \frac{A - s_p (A)}{p - 1}$, we obtain the equalities in \eqref{eq0thm1} and \eqref{eq00thm1}. By Lemma \ref{lemma1}(ii), $p^a \equiv 1\pmod{z(p)}$. Then $r=s$.
		
		\noindent{\bf Case 1.} $p \nmid n $. Let $m=\left\lfloor\frac{p^a n}{z(p)}\right\rfloor$ and $k = \left\lfloor\frac{n}{z(p)}\right\rfloor$. Then we obtain by Theorem \ref{theorem2}(iii) that 
		\begin{equation}\label{eq1thm1}
			\nu_p \left( {p^a n \choose n}_F \right) = \nu_p \left( {m \choose k} \right) = \nu_p(m!) - \nu_p (k!) - \nu_p ((m-k)!).
		\end{equation}
		By Lemma \ref{lemma1}(ii) and Theorem \ref{theorem1}, we see that if $p \equiv \pm 1 \pmod{5}$, then $p \equiv 1 \pmod{z(p)}$ and
		\begin{equation}\label{eq2thm1}
			\nu_p (m!) = \nu_p\left( \left\lfloor \frac{n p^a}{z(p)} \right\rfloor! \right) = \frac{n(p^a - 1)}{z(p)(p-1)} - a\left\{ \frac{n}{z(p)} \right\} + \nu_p\left( k ! \right),
		\end{equation}
		and if $p \equiv \pm 2 \pmod{5}$ and $a$ is even, then $p \equiv -1 \pmod{z(p)}$ and
		\begin{equation}\label{eq3thm1}
			\nu_p (m!) = \frac{n(p^a - 1)}{z(p)(p-1)} - \frac{a}{2}[s \neq 0] + \nu_p\left( k ! \right).
		\end{equation}
		Since $z(p) \mid p^a - 1$ and $p\nmid n$, $A = \frac{n(p^a -1)}{z(p)}$. Therefore
		\begin{equation}\label{eq3.5thm1}
			m-k = \left\lfloor\frac{p^a n}{z(p)}\right\rfloor - \left\lfloor\frac{n}{z(p)}\right\rfloor = \frac{p^a n - r}{z(p)} - \frac{n - s}{z(p)} = \frac{n(p^a - 1)}{z(p)}=A.
		\end{equation}
		Substituting \eqref{eq2thm1}, \eqref{eq3thm1}, and \eqref{eq3.5thm1} in \eqref{eq1thm1}, we obtain \eqref{eq0thm1} and \eqref{eq00thm1}. 
		
		\noindent{\bf Case 2.} $p \mid n $. Let $n = p^b \ell$ where $p \nmid \ell$, $m=\left\lfloor\frac{\ell p^a}{z(p)}\right\rfloor$, and $k=\left\lfloor\frac{\ell}{z(p)}\right\rfloor$. Since $r=s$, we obtain by Theorem \ref{theorem3} that $\nu_p \left( {p^a n \choose n}_F \right)$ is equal to
		\begin{equation}\label{eq4thm1}
			\nu_p \left( {\ell p^{a + b} \choose \ell p^{b}}_F \right) = \nu_p \left( {m \choose k} \right) = \nu_p (m!) - \nu_p (k!) - \nu_p ((m-k)!).
		\end{equation}
		Since $\gcd (p,z(p)) = 1$, we see that $\ell \equiv 0 \pmod{z(p)} \Leftrightarrow n \equiv 0 \pmod{z(p)} \Leftrightarrow s = 0$. Similar to Case 1, we have $\nu_p(m!) = \frac{\ell(p^a - 1)}{z(p)(p-1)} - a\left\{ \frac{\ell}{z(p)} \right\} + \nu_p (k!)$ if $p \equiv \pm 1 \pmod{5}$, $\nu_p (m!) = \frac{\ell(p^a - 1)}{z(p)(p-1)} - \frac{a}{2}[s \neq 0] + \nu_p\left( k!\right)$ if $p \equiv \pm 2 \pmod{5}$ and $a$ is even, $\ell p^a \equiv \ell \pmod{z(p)}$, $A = \frac{\ell (p^a - 1)}{z(p)}$, and $m-k = A$. So \eqref{eq4thm1} leads to \eqref{eq0thm1} and \eqref{eq00thm1}. This proves (i) and (ii).
		
		To prove (iii), suppose that $p \equiv \pm 2 \pmod{5}$ and $a$ is odd. By Lemma \ref{lemma1}(ii), $p \equiv -1 \pmod{z(p)}$. In addition, $\frac{p^a - 1}{p - 1} = p^{a-1} + p^{a-2} + \ldots + 1 \equiv 1 \pmod{z(p)}$. We divide the consideration into two cases.
		
		\noindent{\bf Case 1.} $p \nmid n $. This case is similar to Case 1 of the previous part. So we apply Theorems \ref{theorem1} and \ref{theorem2}(iii). Let $m = \left\lfloor\frac{p^a n}{z(p)}\right\rfloor$ and $k = \left\lfloor\frac{n}{z(p)}\right\rfloor$. Then
		\begin{equation*}
			\nu_p(m!) = \frac{n (p^a - 1)}{z(p)(p-1)} - \frac{a-1}{2}[s \neq 0] - \left\{ \frac{n}{z(p)} \right\} + \nu_p (k!),
		\end{equation*}
		\begin{equation*}
			m - k = \frac{p^a n -r}{z(p)} - \frac{n-s}{z(p)} = \frac{n(p^a-1) - (r-s)}{z(p)},
		\end{equation*}
		\begin{equation*}
			A = \left\lfloor \frac{np^a-r}{z(p)} - \frac{n - s}{z(p)} + \frac{r-s}{z(p)} \right\rfloor= m - k + \left\lfloor \frac{r-s}{z(p)} \right\rfloor.
		\end{equation*}
		Since $\frac{p^a - 1}{p - 1} \equiv 1 \pmod{z(p)}$, $\frac{n(p^a - 1)}{p-1} \equiv n \pmod{z(p)}$. This implies that $\left\{ \frac{n(p^a - 1)}{z(p)(p-1)} \right\} = \left\{ \frac{n}{z(p)} \right\}$. Therefore
		\begin{align*}
			\nu_p(m!) &= \left\lfloor \frac{n (p^a - 1)}{z(p)(p-1)} \right\rfloor - \frac{a-1}{2}[s \neq 0] + \nu_p (k!) = \left\lfloor \frac{A}{p-1} \right\rfloor - \frac{a-1}{2}[s \neq 0] + \nu_p (k!).
		\end{align*}
		From the above observation, if $r\geq s$, then $A=m-k$ and
		\begin{equation*}
			\nu_p \left( {p^a n \choose n}_F \right) = \nu_p \left( {m \choose k} \right) = \left\lfloor \frac{A}{p-1} \right\rfloor - \frac{a-1}{2}[s \neq 0] - \nu_p (A!),
		\end{equation*}
		which leads to \eqref{eq000thm1}. If $r < s$, then $A = m - k - 1$, $\left\lfloor \frac{p^a n - n + z(p)}{z(p)} \right\rfloor = A + 1$, and $\nu_p \left( {p^a n \choose n}_F \right)$ is equal to
		\begin{align*}
			& \left\lfloor \frac{A}{p-1} \right\rfloor - \frac{a-1}{2}[s \neq 0] - \nu_p ((A+1)!) + \nu_p(A+1) + \nu_{p} (F_{z(p)}) \\
			&= \left\lfloor \frac{A}{p-1} \right\rfloor - \frac{a-1}{2}[s \neq 0] - \nu_p (A!) + \nu_p (F_{z(p)}),
		\end{align*}
		which is the same as \eqref{eq000thm1}.
		
		\noindent{\bf Case 2.} $p \mid n $. Let $n=p^b \ell$ where $p \nmid \ell$, $m = \left\lfloor \frac{\ell p^a}{z(p)} \right\rfloor$, and $k = \left\lfloor \frac{\ell}{z(p)} \right\rfloor$. Similar to Case 1, $ s = 0 \Leftrightarrow \ell \equiv 0 \pmod{z(p)}$. In addition, $\frac{\ell(p^a - 1)}{p - 1} \equiv \ell \pmod{z(p)}$, and so we obtain by Theorem \ref{theorem1} that $\nu_p (m!) = \left\lfloor \frac{A}{p-1} \right\rfloor - \frac{a-1}{2}[s \neq 0] + \nu_p (k!)$. The calculation of $\nu_p \left( {p^a n \choose n}_F \right) = \nu_p \left( {\ell p^{a+b} \choose \ell p^b}_F \right)$ is done by the applications of Theorem \ref{theorem3} and is divided into several cases. Suppose $r = s$. Then $p^{a+b} \ell \equiv p^a n \equiv r \equiv s \equiv n \equiv p^b \ell \pmod{z(p)}$. Since $(p,z(p)) = 1$, this implies $\ell p^a \equiv \ell \pmod{z(p)}$. Therefore $A = \left\lfloor \frac{\ell p^a - \ell}{z(p)} \right\rfloor = \frac{\ell p^a - \ell}{z(p)} = m - k$ and 
		\begin{equation*}
			\nu_p \left( {p^a n \choose n}_F \right) = \nu_p \left( {m \choose k} \right) = \nu_p (m!) - \nu_p (k!) - \nu_p ((m-k)!),
		\end{equation*}
		which is \eqref{eq000thm1}. Obviously, if $\ell \equiv 0 \pmod{z(p)}$, then $r=s$, which is already done. So from this point on, we assume that $r \neq s$ and $\ell \not\equiv 0 \pmod{z(p)}$. Recall that $p \equiv -1 \pmod{z(p)}$ and $a$ is odd. So if $b$ is odd, then 
		\begin{equation*}
			r \equiv np^a \equiv -n \equiv -p^b \ell \equiv \ell \pmod{z(p)}, \ \ s \equiv n \equiv p^b \ell \equiv -\ell \equiv \ell p^a \pmod{z(p)}, \ \  \text{and} 
		\end{equation*}
		\begin{equation*}
			A = \left\lfloor \frac{\ell p^a - s}{z(p)} - \frac{\ell - r}{z(p)} + \frac{s-r}{z(p)} \right\rfloor = \frac{\ell p^a - s}{z(p)} - \frac{\ell - r}{z(p)} + \left\lfloor \frac{s - r}{z(p)} \right\rfloor = m - k + \left\lfloor \frac{s - r}{z(p)} \right\rfloor.
		\end{equation*} 
		Similarly, if $b$ is even, then $r = \ell p^a \bmod{z(p)}$, $s = \ell \bmod{z(p)}$, and $A = m - k + \left\lfloor \frac{r - s}{z(p)} \right\rfloor$. Let $R = \left\lfloor \frac{A}{p-1} \right\rfloor - \frac{a-1}{2}[s \neq 0] - \nu_p (A!) + \delta$ be the quantity in \eqref{eq000thm1}. From the above observation and the application of Theorem \ref{theorem3}, we obtain $\nu_p \left( {p^a n \choose n}_F \right)$ as follows. If $r > s$ and $b$ is even, then $A = m - k$ and
		\begin{equation*}
			\nu_p \left( {p^a n \choose n}_F \right) = \frac{b}{2} + \nu_p \left( {m \choose k} \right) = \frac{b}{2} + \left\lfloor \frac{A}{p-1} \right\rfloor - \frac{a-1}{2}[s \neq 0] - \nu_{p}(A!) = R.
		\end{equation*}
		If $r > s$ and $b$ is odd, then $A = m - k - 1$ and
		\begin{align*}
			\nu_p \left( {p^a n \choose n}_F \right) &= \frac{b + 1}{2} + \nu_p (A+1) + \nu_p \left( {m \choose k} \right) \\
			&= \frac{b+1}{2} + \left\lfloor \frac{A}{p-1} \right\rfloor - \frac{a-1}{2}[s \neq 0] - \nu_p (A!) = R.
		\end{align*}
		If $r < s$ and $b$ is even, then $A = m - k - 1$ and 
		\begin{align*}
			\nu_p \left( {p^a n \choose n}_F \right) &= \frac{b}{2} + \nu_p \left( F_{z(p)} \right) + \nu_p (A+1) + \nu_p \left( {m \choose k} \right) \\
			&= \frac{b}{2} + \nu_p (F_{z(p)}) + \left\lfloor \frac{A}{p-1} \right\rfloor - \frac{a-1}{2}[s \neq 0] - \nu_p (A!) = R.
			\end{align*}
		If $r < s$ and $b$ is odd, then $A = m - k$ and 
		\begin{align*}
			\nu_p \left( {p^a n \choose n}_F \right) &= \frac{b-1}{2} + \nu_p \left( F_{z(p)} \right) + \nu_p \left( {m \choose k} \right) \\
			&= \frac{b-1}{2} + \nu_p (F_{z(p)}) + \left\lfloor \frac{A}{p-1} \right\rfloor - \frac{a-1}{2}[s \neq 0] - \nu_p (A!) = R.
		\end{align*}
		This completes the proof.
	\end{proof}

	In the next two corollaries, we give some characterizations of the integers $n$ such that ${p^a n \choose n}_F$ is divisible by $p$.

	\begin{corollary}
		Let $p$ be a prime and let $a$ and $n$ be positive integers. If $n \equiv 0 \pmod{z(p)}$, then $p \mid {p^a n \choose n}_F$.
	\end{corollary}

	\begin{proof}
		We first consider the case $p \neq 2,5$. Assume that $n \equiv 0 \pmod{z(p)}$ and $r$, $s$, $A$, and $\delta$ are as in Theorem \ref{maintheorem1}. Then $\frac{n}{p^{\nu_p(n)}z(p)}, \frac{A}{p-1} \in \mathbb{Z}$, $r = s = 0$, and $\delta = 0$. Every case in Theorem \ref{maintheorem1} leads to  $\nu_p \left( {p^a n \choose n}_F \right) = \frac{s_p (A)}{p-1} > 0$, which implies $p \mid {p^a n \choose n}_F$. If $p = 5$, then the result follows immediately from Theorem \ref{theoremp5}. If $p = 2$, then every case of Theorem \ref{maintheorem0} leads to $\nu_2 \left( {2^a n \choose n}_F \right) \geq s_2 (A) > 0$, which implies the desired result.
	\end{proof}

	\begin{corollary}
		Let $p \neq 2,5$ be a prime and let $a$, $n$, $r$, $s$, and $A$ be as in Theorem \ref{maintheorem1}. Assume that $p \equiv \pm 2 \pmod{5}$ and $n \not \equiv 0 \pmod{z(p)}$. Then the following statements hold.
		\begin{itemize}
			\item[(i)] Assume that $a$ is even. Then $p \mid {p^a n \choose n}_F$ if and only if $s_p (A) > \frac{a}{2} (p - 1)$.
			\item[(ii)] Assume that $a$ is odd and $p \nmid n$. If $r < s$, then $p \mid {p^a n \choose n}_F$. If $r \geq s$, then $p \mid {p^a n \choose n}_F$ if and only if $s_p (A) \geq \frac{a + 1}{2} (p - 1)$.
			\item[(iii)] Assume that $a$ is odd and $p \mid n$. If $r\neq s$, then $p \mid {p^a n \choose n}_F$. If $r = s$, then $p \mid {p^a n \choose n}_F$ if and only if $s_p (A) \geq \frac{a + 1}{2} (p - 1)$.
		\end{itemize}
	\end{corollary}

	\begin{proof}
		We use Lemmas \ref{Legendre} and \ref{floorlemma} repeatedly without reference. For (i), we obtain by \eqref{eq00thm1} that
		\begin{equation*}
			\nu_p \left( {p^a n \choose n}_F \right) = \frac{s_p (A)}{p - 1} - \frac{a}{2}, \text{ which is positive if and only if $s_p (A) > \frac{a}{2}(p-1)$.}
		\end{equation*}
		This proves (i). To prove (ii) and (iii), we let $\delta$ be as in Theorem \ref{maintheorem1} and divide the consideration into two cases.
		
		\noindent{\bf Case 1.} $p \nmid n$. If $r < s$, then we obtain by Theorem \ref{theorem2}(iii) that $\nu_p \left( {p^a n \choose n}_F \right) \geq \nu_p (F_{z(p)}) \geq 1$. Suppose $r \geq s$. Then $\delta = 0$ and \eqref{eq000thm1} is 
		\begin{equation*}
			\left\lfloor \frac{A}{p - 1} \right\rfloor - \frac{a - 1}{2} - \nu_p (A !) = \left\lfloor \frac{A}{p - 1} \right\rfloor - \frac{a - 1}{2} - \frac{A - s_p (A)}{p-1} = \frac{s_p (A)}{p - 1} - \left\{ \frac{A}{p - 1} \right\} - \frac{a - 1}{2}.
		\end{equation*}
		If $s_p(A) \geq \frac{a+1}{2}(p - 1)$, then \eqref{eq000thm1} implies that
		\begin{equation*}
			\nu_p \left( {p^a n \choose n}_F \right) \geq 1 - \left\{ \frac{A}{p - 1} \right\} > 0.
		\end{equation*}
		Similarly, if $s_p (A) < \frac{a+1}{2}(p-1)$, then $\nu_p \left( {p^a n \choose n}_F \right) < 1 - \left\{ \frac{A}{p-1} \right\} \leq 1$. This proves (ii).
		
		\noindent{\bf Case 2.} $p \mid n$. We write $n = p^b \ell$ where $p \nmid \ell$. Then $b \geq 1$. Recall that $\nu_p (F_{z(p)}) \geq 1$. If $r \neq s$, then Theorem \ref{theorem3} implies that $\nu_p \left( {p^a n \choose n} \right) \geq \frac{b}{2}$ if $b$ is even and it is $\geq \frac{b + 1}{2}$ if $b$ is odd. In any case, $\nu_p \left( {p^a n \choose n}_F \right) \geq 1$. So $p \mid {p^a n \choose n}_F$. If $r = s$, then $\delta = 0$ and we obtain as in Case 1 that $p \mid {p^a n \choose n}_F$ if and only if $s_p (A) \geq \frac{a + 1}{2} (p - 1)$. This proves (iii).
	\end{proof}

	\begin{corollary}
		Let $p \neq 2,5$ be a prime and let $A = \frac{n(p-1)}{p^{\nu_p (n)}z(p)}$. Assume that $p \equiv \pm 1 \pmod{5}$. Then $p \mid {pn \choose n}_F$ if and only if $s_p (A) \geq p-1$.
	\end{corollary}

	\begin{proof}
		We remark that by Lemma \ref{lemma1}(ii), $A$ is an integer.  Let $x = \frac{n}{p^{\nu_p (n)}z(p)}$. We apply Theorem \ref{maintheorem1}(i) with $a = 1$. If $s_p (A) \geq p-1$, then \eqref{eq0thm1} implies that $\nu_p \left( {pn \choose n}_F \right) \geq 1 - \{x\} > 0$. If $s_p (A) < p - 1$, then	$\nu_p \left( {pn \choose n}_F \right) < 1 - \{x\} \leq 1$. This completes the proof.
	\end{proof}

	\section*{Acknowledgments}
	
	This research was jointly supported by the Thailand Research Fund and the Faculty of Science Silpakorn University, grant number RSA5980040.

	\medskip
	
	\noindent MSC2010: 11B39, 11B65, 11A63
	
\end{document}